\documentclass[11pt, reqno]{article}
\usepackage{amsmath,amssymb,amsfonts,amsthm}
\usepackage{color}

\usepackage[colorlinks=true,linkcolor=blue,citecolor=blue,urlcolor=blue,pdfborder={0 0 0}]{hyperref}
\usepackage{cleveref}

\usepackage{caption}
\usepackage{thmtools}

\usepackage{mathrsfs}

\crefname{theorem}{Theorem}{Theorems}
\crefname{thm}{Theorem}{Theorems}
\crefname{lemma}{Lemma}{Lemmas}
\crefname{lem}{Lemma}{Lemmas}
\crefname{remark}{Remark}{Remarks}
\crefname{prop}{Proposition}{Propositions}
\crefname{defn}{Definition}{Definitions}
\crefname{corollary}{Corollary}{Corollaries}
\crefname{conjecture}{Conjecture}{Conjectures}
\crefname{question}{Question}{Questions}
\crefname{chapter}{Chapter}{Chapters}
\crefname{section}{Section}{Sections}
\crefname{figure}{Figure}{Figures}

\theoremstyle{plain}
\newtheorem{thm}{Theorem}[section]

\newtheorem{lemma}[thm]{Lemma}
\newtheorem{theorem}[thm]{Theorem}

\newtheorem{corollary}[thm]{Corollary}

\theoremstyle{definition}

\newtheorem{example}[thm]{Example}

\theoremstyle{remark}
\newtheorem{remark}[thm]{Remark}

\numberwithin{equation}{section}

\newcommand{\R}{\mathbb R}
\newcommand{\Z}{\mathbb Z}
\newcommand{\N}{\mathbb N}

\newcommand{\cP}{\mathcal P}

\newcommand{\sE}{\mathscr E}

\newcommand{\sR}{\mathscr R}

\usepackage[margin=1.07in]{geometry}
\usepackage{extarrows}
\usepackage{titling}
\usepackage{commath}
\usepackage{wasysym}
\usepackage{mathtools}
\usepackage{titlesec}
\newcommand{\eps}{\varepsilon}
\usepackage{bbm}
\usepackage{setspace}
\usepackage{enumitem}

\usepackage[textsize=tiny]{todonotes}

\usepackage{cite}

\newcommand{\bP}{\mathbf P}
\newcommand{\bE}{\mathbf E}

\pretitle{\begin{flushleft}\Large}
\posttitle{\par\end{flushleft}\vskip 0.5em
}
\preauthor{\begin{flushleft}}
\postauthor{
\\
\vspace{0.5em}
\footnotesize{Statslab, DPMMS, University of Cambridge.\\ Email: 
\href{mailto:t.hutchcroft@maths.cam.ac.uk}{t.hutchcroft@maths.cam.ac.uk}}
\end{flushleft}}
\predate{\begin{flushleft}}
\postdate{\par\end{flushleft}}
\title{\bf Transience and recurrence of sets for branching random walk via non-standard stochastic orders}

\renewenvironment{abstract}
 {\par\noindent\textbf{\abstractname.}\ \ignorespaces}
 {\par\medskip}

\titleformat{\section}
  {\normalfont\large \bf}{\thesection}{0.4em}{}

\author{{\bf Tom Hutchcroft}}

\begin{document}

\date{\small{\today}}

\maketitle

\setstretch{1.1}

\begin{abstract}
We study how the recurrence and transience of space-time sets for a branching random walk on a graph depends on the offspring distribution. Here, we say that a space-time set $A$ is \emph{recurrent} 
if it is visited infinitely often almost surely on the event that the branching random walk survives forever, and say that $A$ is \emph{transient} if it is visited at most finitely often almost surely. 
We prove that if $\mu$ and $\nu$ are supercritical offspring distributions with means $\bar \mu < \bar \nu$ then every space-time set that is recurrent with respect to the offspring distribution $\mu$ is also recurrent with respect to the offspring distribution $\nu$ and similarly that every space-time set that is transient with respect to the offspring distribution $\nu$ is also transient with respect to the offspring distribution $\mu$.
To prove this, we introduce a new order on probability measures that we call the \emph{germ order} and prove more generally that the same result holds whenever $\mu$ is smaller than $\nu$ in the germ order. Our work is inspired by the work of Johnson and Junge (AIHP 2018), who used related stochastic orders to study the frog model.
\end{abstract}

\section{Introduction}

Let $P$ be the transition matrix of a Markov chain on a countable state space $S$ and let $\mu$ be an \emph{offspring distribution}, i.e., a probability measure on $\{0,1,\ldots\}$. The \textbf{branching Markov process} associated to the pair $(P,\mu)$  is a Markov process taking values in the space of finitely-supported functions $S\to \{0,1,2,\ldots\}$, where we think of elements of this space as encoding the number of particles occupying each state $S$. We begin with a single particle occupying some state $x\in S$. At each step of the process, each particle splits into a random number of new particles with law $\mu$, each of which immediately  performs an independent step of the underlying Markov chain with transition matrix $P$.
Branching Markov processes can also be described equivalently as Markov chains indexed by Galton-Watson trees \cite{MR1258875,MR1254826}. 
When $P$ is the transition matrix of simple random walk on a graph, the branching Markov process is referred to as \textbf{branching random walk}. Even if one is only interested in branching random walks, one is often led to consider the space-time version of the branching random walk, which is itself a branching Markov process.


It is a classical fact \cite[Chapter 5]{LP:book} that if $\mu(1)<1$ then the branching Markov process has a positive probability to survive forever if and only if the mean $\bar\mu$ of  $\mu$ satisfies $\bar \mu > 1$, in which case we say $\mu$ is \textbf{supercritical}. Many further questions of interest in the study of branching random walk, arising e.g. in the study of the asymptotics of the maximum displacement \cite{MR370721,MR400438,MR464415,MR1797310,MR3444654} or the intersections of  two independent branching random walks \cite{MR2914859,MR3333735,MR3644010,MR4146533}, can be formulated in terms of whether some \emph{subset} $A$ of $S$ or $S \times \Z$ is visited infinitely often by the branching Markov process or space-time
 branching Markov process as appropriate. See \cref{sec:examples} and e.g.\ \cite{MR1349163} for examples. 
 The case in which $A$ is a singleton is very well understood, see \cite{MR1254826,MR2426846,MR2284404}.
In this note, we study the effect of changing the offspring distribution $\mu$ on the transience and recurrence of such subsets $A \subseteq S$.

Let us first write down some relevant definitions. Fix a transition matrix $P$ on a countable state space $S$. For each offspring distribution $\mu$ and $x\in S$, let $\bP_x^\mu=\bP_x^{P,\mu}$ be the law of the associated branching Markov process $(B_n)_{n\geq 0}$ started with a single particle at the state $x$. For each subset $A$ of the state space $S$, we write $L(A)= \sum_{a\in A}\sum_{n\geq 0} B_n(a)$ for the total number of particles that ever visit $A$. 
Let $\Omega_\infty$ be the event that the branching random walk survives forever. We say that a set $A \subseteq S$ is $\mu$\textbf{-recurrent} if $\bP^\mu_x(\Omega_\infty) >0$ and
\[
\bP_x^{\mu}(L(A)=\infty \mid \Omega_\infty) = 1 \qquad \text{ for every $x\in S$}
\]
and similarly that $A$ is $\mu$\textbf{-transient} if
\[
 \bP_x^{\mu}(L(A)=\infty) = 0 \qquad \text{ for every $x\in S$.}
\]
Note that $\mu$-transience and $\mu$-recurrence are mutually exclusive but not mutually exhaustive: For example, for branching random walk on a binary tree with  mean offspring $\bar \mu = 1+\eps$, the set of vertices descended from the left-most child of the root has probability strictly between $0$ and $1$ to be visited infinitely often on the event $\Omega_\infty$ and is therefore neither recurrent nor transient by these definitions when $\eps>0$ is sufficiently small.

The simplest form of our theorem is as follows. 

\begin{thm}
\label{thm:simple}
Let $P$ be a transition matrix on a countable state space $S$ and let $\mu$ and $\nu$ be supercritical offspring distributions with means $\bar \mu < \bar \nu$. Then every $\nu$-transient set is $\mu$-transient and every $\mu$-recurrent set is $\nu$-recurrent.
\end{thm}

This theorem has the following immediate corollary.

\begin{corollary}
Let $P$ be a transition matrix on a countable state space $S$. For each $A \subseteq S$ there exist $\lambda_t(A) \leq \lambda_r(A) \in [1,\infty]$ such that the following hold for each offspring distribution $\mu$ with mean $\bar\mu$:
\begin{enumerate}
	\item If $\bar \mu < \lambda_t(A)$ then $A$ is $\mu$-transient.
	\item If $\bar \mu > \lambda_t(A)$ then $A$ is not $\mu$-transient.
	\item If $\bar \mu < \lambda_r(A)$ then $A$ is not $\mu$-recurrent.
	\item If $\bar \mu > \lambda_r(A)$ then $A$ is $\mu$-recurrent.
\end{enumerate}
\end{corollary}

Related results for branching random walks on $\R$ with i.i.d.\ increments have been proven under a second moment assumption by Pemantle and Peres \cite{MR1349163}. 
 Their proof is completely different to ours and deals with \emph{ray-transience} (i.e., the non-existence of infinite lines of descent intersecting $A$ infinitely often) rather than transience as we define it here.


\textbf{Non-standard stochastic orders.}
We will deduce \cref{thm:simple} as a special case of a more general theorem involving a new partial ordering of the set of offspring distributions that we call the \emph{germ ordering}. Our arguments are inspired by the work of Johnson and Junge \cite{MR3795075} (see also \cite{MR3706732}), who studied the dependence of the \emph{frog model} on the its particle distribution using a non-standard stochastic order they called the \emph{pgf ordering}. 
 Given a probability measure $\mu$ on $[0,\infty)$,  the probability generating function (pgf) of $\mu$ is defined to be $\cP_\mu(t) = \int_0^\infty t^x \dif \mu(x)$ for each $t \in [0,1]$. The \textbf{pgf ordering} $\leq_\mathrm{pgf}$ on the set of probability measures on $[0,\infty]$ is defined by
\[
\Bigl(\mu \leq_\mathrm{pgf} \nu\Bigr) \iff \Bigl( \cP_\mu(t)  \geq \cP_\nu(t) \text{ for every $0\leq t\leq 1$}\Bigr)
\]
for each two probability measures $\mu$ and $\nu$ on $[0,\infty)$. The pgf order was also previously studied in the context of signal processing and wireless networks \cite{tepedelenlioglu2011applications,MR3160309}, where it is known as \emph{Laplace transform order}. The pgf order is just one example from among a number of non-standard stochastic orders that have recently found applications to many problems throughout the literature, which we now briefly review. Recall that the \textbf{standard stochastic order} (a.k.a. stochastic domination) is defined by $\mu \leq_\mathrm{st} \nu$ if and only if $\int f(t) \dif \mu(t) \leq \int f(t) \dif \nu(t)$ for every increasing function $f$, and is used very widely throughout probability theory. The \textbf{increasing concave (icv) order} is defined by $\mu \leq_\mathrm{icv} \nu$ if and only if $\int f(t) \dif \mu(t) \leq \int f(t) \dif \nu(t)$ for every increasing concave function $f$, and has recently found important applications to first passage percolation~\cite{MR1925450,MR1202515}, reinforced random walks \cite{poudevigne2019monotonicity}, the frog model \cite{MR3795075}, and random walk in random potential \cite{MR1620370}.
See \cite{MR3795075} for a detailed discussion of these orders and how they compare to each other.


\textbf{The germ order}.
 We now introduce the new order on probability distributions that we consider.
We define the \textbf{germ order} $\leq_\mathrm{germ}$ on the set of probability measures on $[0,\infty)$ by
\[
\Bigl(\mu \leq_\mathrm{germ} \nu \Bigr) \iff
\Bigl( \text{there exists $\eps>0$ such that }  \cP_\mu(t) \geq \cP_\nu(t) \text{ for every $1-\eps \leq t\leq 1$}
\Bigr)
\]
for each two probability measures $\mu$ and $\nu$ on $[0,\infty)$. The germ order refines the pgf order in the sense that any two probability measures $\mu$ and $\nu$ satisfying $\mu \leq_\mathrm{pgf} \nu$ must also trivially satisfy $\mu \leq_\mathrm{germ} \nu$. Indeed, we have moreover that
\[
\Bigl(\mu \leq_\mathrm{st} \nu \Bigr) \Rightarrow \Bigl(\mu \leq_\mathrm{icv} \nu \Bigr) \Rightarrow \Bigl(\mu \leq_\mathrm{pgf} \nu \Bigr) \Rightarrow \Bigl(\mu \leq_\mathrm{germ} \nu \Bigr).
\]
However, the germ order is much finer than the pgf order in the sense that many more pairs of measures can be compared. Indeed, if $\mu$ has finite $p$th moment $\sum_{n\geq 0} n^p \mu(n)$ for some $p\geq 1$ then we have by a standard computation that
\begin{equation}
\label{eq:Taylor}
\cP_\mu(e^{-\eps})= 1+\sum_{k=1}^p \frac{(-\eps)^k}{k!} \sum_{n\geq 0} n^k \mu(n) + o\left(\eps^p\right)\qquad \text{ as $\eps \downarrow 0$}.
\end{equation}
It follows in particular that if the means of $\mu$ and
$\nu$ satisfy a strict inequality 
$\bar \mu < \bar \nu$ then $\mu \leq_\mathrm{germ} \nu$. Similarly, if $\mu$ and $\nu$ have the same finite mean $\bar \mu = \bar \nu$ but the second moment of $\mu$ is strictly larger than that of $\nu$ then $\mu \leq _\mathrm{germ} \nu$. (Be careful to note that having a large second moment makes $\mu$ \emph{smaller} with respect to the germ order when the first moment is fixed!) Similar statements hold for higher moments, so that, for example, if $\mu$ and $\nu$ have the same first and second moments but the third moment of $\nu$ is strictly larger than that of $\mu$ then $\mu \leq_\mathrm{germ} \nu$. Moreover, when restricted to the class of measures with, say, moments of all orders satisfying Carleman's condition $\sum_{p\geq 1} (p$th moment$)^{-1/2p}=\infty$, the germ order is equivalent to the lexicographical order associated to the alternating sequences (first moment, $-($second moment$)$, third moment, $-($fourth moment$)$, $\ldots$\,). It follows that the germ order is a total order when restricted to this class of measures, so that if $\mu$ and $\nu$ both have moments of all orders satisfying Carleman's condition then either $\mu \leq_\mathrm{germ} \nu$ or $\nu \leq_\mathrm{germ} \mu$. This follows \eqref{eq:Taylor} together with Carleman's classical theorem  \cite{akhiezer2020classical} that if $\mu$ and $\nu$ are probability measures whose moment sequences are equal and satisfy his condition then $\mu=\nu$. More generally, one expects to be able to compare most non-pathological pairs of measures arising in examples. Among the probability measures on $[0,\infty)$ with mean $m<\infty$, the atom at $m$ is maximal with respect to the germ order. Similarly, among the probability measures on $\{0,1,\ldots\}$ with mean $m<\infty$, the unique such measure that is supported on $\{\lfloor m \rfloor,\lceil m \rceil\}$ is maximal with respect to the germ order.


We are now ready to state the most general version of our theorem. 

\begin{thm}
\label{thm:germ}
Let $P$ be a transition matrix on a countable state space $S$ and let $\mu$ and $\nu$ be supercritical offspring distributions such that $\mu \leq_\mathrm{germ}  \nu$. Then every $\nu$-transient set is $\mu$-transient and every $\mu$-recurrent set is $\nu$-recurrent.
\end{thm}

\begin{remark}
It is a theorem of Johnson and Rolla \cite{MR3962479} that the analogue of \cref{thm:simple} does \emph{not} hold for the recurrence of the frog model. This suggests that the germ order is of limited applicability to that model. It would be interesting to see if the germ order has applications to some of the other models mentioned above or, say, to activated random walk \cite{rolla2020activated,rolla2017universality}.
\end{remark}

\cref{thm:germ} implies various related statements concerning alternative notions of transience and recurrence. Let us now highlight one such variation that will be useful when studying the Minkowski dimensions of limit sets in Example~\ref{example:dimension}.

\begin{theorem}
\label{thm:variations}
Let $P$ be a transition matrix on a countable state space $S$, let $\{A_i : i \in I\}$ be a countable collection of $S$, and let $\mu \leq_\mathrm{germ} \nu$ be two offspring distributions. Let $(B_n)_{n\geq 0}$ be a branching Markov process and let $I_\infty=\{ i \in I : A_i$ is visited infinitely often by $B\}$.
\begin{enumerate}
\item If $I_\infty$ is infinite $\bP^\mu_x$-a.s.\ on the event $\Omega_\infty$ for every $x\in S$ then $I_\infty$ is also infinite $\bP^\nu_x$-a.s.\ on the event $\Omega_\infty$ for every $x\in S$.
\item If $I_\infty$ is finite $\bP^\nu_x$-a.s.\ for every $x\in S$ then $I_\infty$ is also finite $\bP^\mu_x$-a.s.\ for every $x\in S$.
\end{enumerate}
\end{theorem}


\section{Proof}


In this section we prove \cref{thm:germ,thm:variations}. As a warm up, we begin by proving a similar result under the stronger assumption that $\mu \leq_\mathrm{pgf} \nu$, which is much more straightforward.

\begin{lemma}
\label{lemma:pgf_induction}
Let $P$ be a transition matrix with countable state space $S$. If $\mu,\nu$ are offspring distributions such that $\mu \leq_\mathrm{pgf} \nu$ then
\[
\bE_{x}^{\mu}\left[e^{-t L(A)}\right] \geq \bE_{x}^{\nu}\left[e^{-t L(A)}\right]
\]
for every $x \in S$, $A \subseteq S$, and $t \geq 0$. 
\end{lemma}

In other words, the laws of the local times at $A$ under the two measures satisfy the same relationship under the pgf order as their associated offspring distributions.

\begin{proof}[Proof of \cref{lemma:pgf_induction}]
Fix $A \subseteq S$. For each $n\geq 0$, let $L_n(A)=\sum_{a\in A}\sum_{m=0}^nB_m(a)$ be the total number of particles that have visited $A$ by time $n$. We will prove by induction on $n$ that
\begin{equation}
\bE_{x}^{\mu}\left[e^{-t L_n(A)}\right] \geq \bE_{x}^{\nu}\left[e^{-t L_n(A)}\right]
\end{equation}
for every $x\in S$, $t\geq 0$, and $n\geq 0$; the result will then follow from the monotone convergence theorem. The base case $n=0$ is trivial since there we have that
\[
\bE_{x}^{\mu}\left[e^{-t L_0(A)}\right] = \bE_{x}^{\nu}\left[e^{-t L_0(A)}\right] = e^{-t}\mathbbm{1}(x \in A) + \mathbbm{1}(x \notin A).
\]
Meanwhile, for $n\geq 1$ we have by the definitions that if $|B_1|=\sum_y B_1(y)$ denotes the number of particles in generation $1$ then
\begin{align}
\bE_{x}^{\mu}\left[e^{-t L_n(A)}\right] 
&= \bE_x^\mu\left[\bE_{x}^{\mu}\left[e^{-t L_n(A)} \;\Big|\; |B_1| \right]  \right]
\nonumber\\
&= \left(e^{-t}\mathbbm{1}(x \in A) + \mathbbm{1}(x \notin A)\right)\sum_{k\geq 0} \mu(k) \left(\sum_{y\in S} P(x,y) \bE_{y}^{\mu}\left[e^{-t L_{n-1}(A)}\right]\right)^k
\label{eq:pgf_recursion}
\end{align}
and applying the definition of the pgf order and induction hypothesis yields that
\begin{align*}
\bE_{x}^{\mu}\left[e^{-t L_n(A)}\right]
&\geq\left(e^{-t}\mathbbm{1}(x \in A) + \mathbbm{1}(x \notin A)\right)\sum_{k\geq 0} \nu(k) \left(\sum_{y\in V} P(x,y) \bE_{u}^{\mu}\left[e^{-t L_{n-1}(A)}\right]\right)^k\\
&\geq\left(e^{-t}\mathbbm{1}(x \in A) + \mathbbm{1}(x \notin A)\right)\sum_{k\geq 0} \nu(k) \left(\sum_{y\in V} P(x,y) \bE_{u}^{\nu}\left[e^{-t L_{n-1}(A)}\right]\right)^k
\nonumber\\
&=\bE_x^\nu\left[\bE_{x}^{\nu}\left[e^{-t L_n(A)} \;\big|\; |B_1| \right]  \right]=\bE_{x}^{\nu}\left[e^{-t L_n(A)}\right]\nonumber
\end{align*}
as required.
\end{proof}

\begin{corollary}
\label{cor:pgf}
Let $P$ be a transition matrix on a countable state space $S$ and let $\mu$ and $\nu$ be supercritical offspring distributions such that $\mu \leq_\mathrm{pgf}  \nu$. Then every $\nu$-transient set is $\mu$-transient and every $\mu$-recurrent set is $\nu$-recurrent.
\end{corollary}

\begin{proof}[Proof of \cref{cor:pgf}]
Fix $A \subseteq S$. We have by monotone convergence and \cref{lemma:pgf_induction} that
\begin{align*}\bP_x^{\mu}(L(A) = \infty) = 1-\lim_{t \downarrow 0} \bE_x^{\mu}\left[e^{-tL(A)}\right] \leq 1-\lim_{t \downarrow 0} \bE_x^{\nu}\left[e^{-tL(A)}\right] = \bP_x^{\nu}(L(A) = \infty)
\end{align*}
for every $x\in S$. It follows immediately that if $A$ is $\nu$-transient then it is $\mu$-transient also. 
Now suppose that $A \subseteq S$ is $\mu$-recurrent. Since $\mu$ is supercritical and $\bP_x^{\mu}(\Omega_\infty)=p>0$ does not depend on $x$, we have that
\[
\bP_x^{\nu}(L(A) = \infty) \geq \bP_x^{\mu}(L(A) = \infty) = \bP_x^{\mu}(\Omega_\infty) = p
\]
for every $x\in S$.
Letting $|B_n| = \sum_{y\in S} B_n(y)$ be the total number of particles in generation $n$, we have by the Markov property of the branching Markov process that
\[\bP_x^{\nu}(L(A) < \infty \mid B_n)
= \prod_{y \in S} \bP_y^{\nu}(L(A) < \infty)^{B_n(y)} \leq (1-p)^{|B_n|}
\]
for every $x\in S$ and $n\geq 1$, so that taking expectations over both sides and sending $n\to\infty$ yields that
\[\bP_x^{\nu}(L(A) < \infty)
 \leq \lim_{n\to\infty} \bE_x^{\nu} \left[ (1-p)^{|B_n|} \right] = \bP_x^{\nu}(\Omega_\infty^c),
\]
where we used that, since $\nu$ is supercritical, $|B_n| \to\infty$  almost surely as $n\to\infty$ on the event $\Omega_\infty$. This implies that $\bP_x^{\nu}(L(A) = \infty \mid \Omega_\infty)=1 $, so that $A$ is $\nu$-recurrent as required.
\end{proof}

We now extend this result from the pgf order to the germ order, which requires a rather more subtle approach. We write $a \wedge b = \min\{a,b\}$ and $a \vee b = \max \{a,b\}$.

\begin{lemma}
\label{lem:germ_recursion}
Let $P$ be a transition matrix on a countable state space $S$ and let $\mu$ and $\nu$ be supercritical offspring distributions with $\mu \leq_\mathrm{germ} \nu$. Then there exists $\eps>0$ such that
\[
\bE_x^{\mu}\left[t^{\mathbbm{1}(L(A)>0)}\right] \geq t \vee \bE_x^{\nu}\left[t^{L(A)}\right]
\]
for every $A\subseteq V$ and $1-\eps \leq t \leq 1$.
\end{lemma}

\begin{proof}[Proof of \cref{lem:germ_recursion}]
Fix $A \subseteq S$.
Let $\alpha<1$ be such that $\mathcal{P}_\mu(t) \geq \mathcal{P}_\nu(t)$ for every $\alpha \leq t \leq 1$. The idea is to analyze an approximate version of the recursion used in \cref{lemma:pgf_induction} but where we do not allow any of the values to drop below $\alpha$, and then relate the fixed points of this recursion to the branching random walk.
For each function $F:S \to [\alpha,1]$ we define $\mathbf{I}_\mu F:S\to [\alpha,1]$
 and $\mathbf{I}_\nu F:S\to [\alpha,1]$ by
\begin{align*}
\mathbf{I}_\mu F(x) &= \left[ \alpha \vee \left(\alpha^{\mathbbm{1}(x\in A)} 
\mathcal{P}_\mu \bigl[PF(x)\bigr]\right) \right] \wedge F(x) \qquad \text{ and }\\
 \mathbf{I}_\nu F(x) &= \left[ \alpha \vee \left(\alpha^{\mathbbm{1}(x\in A)} 
\mathcal{P}_\nu \bigl[PF(x)\bigr]\right) \right] \wedge F(x)
\end{align*}
for each $x\in S$, where we define $PF(x) := \sum_{y\in S}P(x,y)F(y)$ for each $x\in S$.
  Since $P$ is a transition matrix, if $F : S \to [\alpha,1]$ then $\alpha \leq PF(x) \leq 1$ for every $x\in S$ also, and it follows by definition of $\alpha$ that
\begin{equation}
\alpha \leq \mathbf{I}_\nu F(x) \leq \mathbf{I}_\mu F(x) \leq F(x) \leq 1
\label{eq:Iproperties}
\end{equation}
for every $F:S\to [\alpha,1]$ and $x\in S$. Moreover, $\mathbf{I}_\mu$ and $\mathbf{I}_\nu$ are both monotone increasing in the sense that if $F,G :S\to [\alpha,1]$ are such that $F \leq G$ pointwise then $\mathbf{I}_\mu F \leq \mathbf{I}_\mu G$ and $\mathbf{I}_\nu F \leq \mathbf{I}_\nu G$ pointwise. 

We consider the sequences of functions $(F_n)_{n\geq 0}$ and $(G_n)_{n\geq 0}$ from $S$ to $[\alpha,1]$ defined recursively by $F_0(x)=G_0(x)=\alpha^{\mathbbm{1}(x\in A)}$ for every $x\in S$ and
\[
 F_{n+1}=\mathbf{I}_\mu F_n, \qquad G_{n+1}=\mathbf{I}_\nu G_n \qquad \text{ for every $n\geq 0$}.
\]
It follows inductively by \eqref{eq:Iproperties} that
$F_n \geq G_n$ pointwise 
for every $n \geq 1$. Moreover, the sequences $F_n$ and $G_n$ are pointwise decreasing, and therefore converge pointwise as $n\to\infty$ to some functions $F_\infty,G_\infty:S\to[\alpha,1]$ such that
\begin{equation}
\label{eq:GPcomparison}
 \alpha \leq G_\infty(x) \leq F_\infty(x) \leq \alpha^{\mathbbm{1}(x\in A)} \text{ for every $x\in S$.}
\end{equation}
 We next claim that
\begin{equation}
\label{eq:Gupper}
G_n(x) \geq \alpha \vee \bE_x^{\nu}\left[\alpha^{L_n(A)}\right] \text{ for every $n\geq 0$ and hence that } G_\infty(x) \geq \alpha \vee \bE_x^{\nu}\left[\alpha^{L(A)}\right]
\end{equation}
 for every $x\in S$. We prove this claim by induction on $n$, the base case $n=0$ being trivial. If this inequality holds for some $n\geq 0$ then, letting $|B_n|=\sum_y B_n(y)$ denote the total number of particles in the $n$th generation of the process, we have that
\begin{align*}
\alpha \vee \bE_x^{\nu}\left[\alpha^{L_{n+1}(A)}\right] &=
\alpha \vee \bE_x^\nu\left[\bE_x^{\nu}\left[\alpha^{L_{n+1}(A)} \;\Big|\; |B_1|\right]\right]
\\&=
 \alpha \vee \left(\alpha^{\mathbbm{1}(x\in A)}\cP_\nu\!\left[ \sum_{y\in S}P(x,y)\bE_y^{\nu}\left[\alpha^{L_{n}(A)}\right] \right]\right)\leq 
 \alpha \vee \left(\alpha^{\mathbbm{1}(x\in A)}\cP_\nu\!\left[ PG_{n}(x) \right]\right)
\end{align*}
for every $x\in S$. We also trivially have that
\begin{equation*}
\alpha \vee \bE_x^{\nu}\left[\alpha^{L_{n+1}(A)}\right] \leq \alpha \vee \bE_x^{\nu}\left[\alpha^{L_{n}(A)}\right] \leq G_n(x),
\end{equation*}
and it follows that
\[
\alpha \vee \bE_x^{\nu}\left[\alpha^{L_{n+1}(A)}\right]
\leq 
 \left[\alpha \vee \left(\alpha^{\mathbbm{1}(x\in A)}\cP_\nu\!\left[ PG_{n}(x) \right]\right)\right] \wedge G_n(x) = \mathbf{I}_\nu G_n(x) = G_{n+1}(x)
\]
for every $x\in S$. This completes the induction step and hence also the proof of \eqref{eq:Gupper}.


To complete the proof, it suffices in light of the inequalities  \eqref{eq:GPcomparison} and \eqref{eq:Gupper} to prove that 
\[F_\infty(x) \leq \bE_x^{\mu}\left[\alpha^{\mathbbm{1}(L(A)>0)}\right]\]
for every $x\in S$. Observe that $\mathbf{I}_\mu$ maps the product space $[\alpha,1]^S$ to itself continuously, so that $F_\infty$ is a fixed point of $\mathbf{I}_\mu$ in the sense that $\mathbf{I}_\mu F_\infty = \lim_{n\to\infty} \mathbf{I}_\mu F_n = \lim_{n\to\infty} F_{n+1} = F_\infty$.
 Let $D = \{x\in S : F_\infty(x)=\alpha\} \supseteq A$ and consider the sequence of functions $(H_n)_{n\geq 0}$ from $S$ to $[\alpha,1]$ defined recursively by 
$H_0(x)=\alpha^{\mathbbm{1}(x \in D)}$ and $H_{n+1}= \mathbf{I}_\mu H_n$ for each $n\geq 0$. Since $F_\infty \leq H_0$, we have by induction that 
\[F_\infty =\mathbf{I}_\mu F_\infty \leq \mathbf{I}_\mu H_{n-1} = H_n \] for every $n\geq 1$. It follows in particular that $H_n(x) \geq F_\infty(x)>\alpha$ for every $n\geq 0$ and $x\notin D$, from which we deduce that the sequence $(H_n)_{n\geq 0}$ must satisfy the recursion
\begin{equation}
\label{eq:H_recursion}
H_{n+1}(x) = \begin{cases}
\alpha & x \in D\\
\mathcal{P}_\mu \left(PH_n(x)\right) \wedge H_n(x) & x \notin D.
\end{cases}
\end{equation}
Indeed, we see by inspecting the definition of $\mathbf{I}_\mu$ that this recursion can only fail to hold at some step $n+1$ if $H_{n+1}(x) = \alpha$ for some $x\notin D$.
Let $E_n(D)$ be the number of particles that occupy $D$ by time $n$ and do not have any ancestors that occupied $D$, and let $E(D)$ be the total number of such particles over all time. (In particular, if the branching Markov process starts at an element of $D$ then the particle at time $0$ is the only such particle, so that $E(D)=1$.) We now claim that $\bE_x^{\mu}[\alpha^{E_n(D)}] = H_n(x)$ for every $n\geq 0$: Indeed, we observe that $H_n'(x):=\bE_x^{\mu}[\alpha^{E_n(D)}]$ satisfies the recursion 
\begin{align*}
H'_{n+1}(x) = \bE_x^{\mu}\left[\bE_x^{\mu}\left[\alpha^{E_{n+1}(D)} \;\Big|\; |B_1|\right]\right] &= \begin{cases}
\alpha & x \in D\\
\mathcal{P}_\mu \left(PH_n'(x)\right) & x \notin D 
\end{cases}
\end{align*}
for every $x\in S$ and $n\geq 0$, and since $H_{n+1}'(x) \leq H_n'(x)$ for every $n\geq 0$ and $x\in S$ it follows moreover that
\begin{align*}
H'_{n+1}(x) &=
\begin{cases}
\alpha & x \in D\\
\mathcal{P}_\mu \left(PH_n'(x)\right) \wedge H_n'(x) & x \notin D 
\end{cases}
\end{align*}
for every $x\in S$ and $n\geq 0$. That is, the two sequences $(H_n)_{n\geq 0}$ and $(H_n')_{n\geq 0}$ satisfy the same recursive relationship between their successive terms.
 Since we also trivially have that $H_0=H_0'$, it follows by induction that $H_n'(x)=\bE_x^{\mu}[\alpha^{E_n(D)}] = H_n(x)$ for every $n\geq 0$ and $x\in S$ as claimed. Since $E(D) \geq 1$ when $L(A)>0$, we may take $n\to \infty$ and deduce by bounded convergence that
\[
F_\infty(x) \leq \lim_{n \to \infty} H_n(x) = \bE_x^{\mu}\left[\alpha^{E(B)}\right] \leq \bE_x^{\mu}\left[ \alpha^{\mathbbm{1}(L(A)>0)}\right]
\]
for every $x\in S$ as claimed.
\end{proof}

We now apply \cref{lem:germ_recursion} to prove \cref{thm:germ}.

\begin{proof}[Proof of \cref{thm:germ}]
We consider the space-time version of the branching Markov process that has state space $S \times \Z$ and transition matrix
\[
P^\mathrm{st}\left((x,m),(y,n)\right) = \begin{cases} P(x,y) & n=m+1\\
0 & \text{otherwise} 
\end{cases}
\qquad \text{ for each $x,y\in S$ and $m,n \in \Z$.}
\]
Note that, for each offspring distribution $\mu$, the branching Markov process $(B^\mathrm{st}_n)_{n\geq 0}$ associated to the pair $(P^\mathrm{st},\mu)$ started at $(x,m)$ can be coupled with the branching Markov process $(B_n)_{n\geq 0}$ associated to $(P,\mu)$ by taking $B_n(y)=B^\mathrm{st}(y,m+n)$ for each $y\in S$ and $n\geq 0$. Moreover, a set $A \subseteq S$ is visited infinitely often by the branching Markov process $(B_n)_{n\geq 0}$ if and only if $A \times \Z$ is visited infinitely often by the space-time branching Markov process $(B^\mathrm{st}_n)_{n\geq 0}$. As such, it suffices to prove that if $\mu \leq_\mathrm{germ} \nu$ then every $\nu$-transient subset of $S \times \Z$ is $\mu$-transient and every $\mu$-recurrent subset of $S \times \Z$ is $\nu$-recurrent. 

First suppose that $A \subseteq S \times \Z$ is $\nu$-transient.
From now on we will write $\bP_{x,m}^{\mu}$ and $\bP_{x,m}^\nu$ for the laws of the space-time branching Markov process $(B^\mathrm{st}_n)_{n\geq 0}$ started at $(x,m)\in S \times \Z$ with offspring distribution $\mu$ or $\nu$ as appropriate.
 Let $A_k = A \cap (S \times [k,\infty))$ for each $k\in \Z$, so that $L(A)=\infty$ if and only if $L(A_k)=\infty$ for infinitely many non-negative $k$, if and only if $L(A_k)>0$ for every $k\in \Z$.
 Applying \cref{lem:germ_recursion} to the space-time process yields that there exists $t<1$ such that
\begin{align*}
\bP^{\mu}_{x,m}(L(A_k)=0)=
\frac{\bE_{x,m}^{\mu}\left[t^{\mathbbm{1}(L(A_k)>0)}\right]-t}{1-t} \geq \frac{t \vee \bE_{x,m}^{\nu}\left[t^{L(A_k)}\right]-t}{1-t}
\end{align*}
for every $(x,m)\in S \times \Z$.
Taking the limit as $k\to\infty$, we deduce by bounded convergence that
\begin{align*}
\bP^{\mu}_{x,m}(L(A)<\infty) &= \lim_{k\to \infty}\bP^{\mu}_{x,m}(L(A_k)=0) \\&\geq
\lim_{k\to\infty} \frac{t \vee \bE_{x,m}^{\nu}\left[t^{L(A_k)}\right]-t}{1-t} = \frac{t\vee\bP^{\nu}_{x,m}(L(A)<\infty)-t}{1-t}
\end{align*}
for every $(x,m)\in S \times \Z$.
Since $A$ is $\nu$-transient the the right hand side is equal to $1$ for every $(x,m)\in S \times \Z$, so that the left hand side is also equal to $1$ for every $(x,m)\in S \times \Z$ and $A$ is $\mu$-transient as claimed.

Now suppose that $A \subseteq S \times \Z$ is $\mu$-recurrent.  We apply \cref{lem:germ_recursion} to obtain that there exists $t<1$ such that
\begin{align*}
  \bP_{x,m}^{\nu}(L(A)=0) &\leq  \bE_{x,m}^{\nu}\left[t^{L(A)}\right] \leq t\vee \bE_{x,m}^{\nu}\left[t^{L(A)}\right]  \leq t\bP_{x,m}^{\mu}(L(A)>0)+\bP_{x,m}^{\mu}(L(A)=0)
\end{align*}
for every $(x,m)\in S \times \Z$.
Letting $p=\bP_{x,m}^{\mu}(\Omega_\infty)>0$ be the survival probability of the $\mu$-branching process, which does not depend on the choice of $(x,m) \in S \times \Z$, we deduce that
\begin{multline*}
 \bP_{x,m}^{\nu}(L(A)=0)  \leq tp + t(1-p)\bP_{x,m}^{\mu}(L(A)>0 \mid \Omega_\infty^c) + (1-p) \bP_{x,m}^{\mu}(L(A)=0 \mid \Omega_\infty^c)\\
 \leq tp+(1-p) < 1
\end{multline*}
for every $(x,m)\in S \times \Z$. It follows by the Markov property that
\[
\bP_{x,m}^{\nu}(L(A)=0 \mid B_n^\mathrm{st}) \leq \prod_{y\in S} \bP_{y,m+n}^\nu(L(A) = 0 )^{B^\mathrm{st}_n(y,m+n)} \leq (1-p+tp)^{|B^\mathrm{st}_n|},
\]
for every $(x,m)\in S \times \Z$ and $n\geq 0$, where, as before, $|B^\mathrm{st}_n|=\sum_{y} B^\mathrm{st}_n(y,m+n)$ denotes the total number of particles at time $n$. Since $\nu$ is supercritical we have that $|B^\mathrm{st}_n| \to \infty$ almost surely as $n\to \infty$ on the event $\Omega_\infty$, so that taking expectations over both sides and sending $n\to\infty$ yields that $\bP_{x,m}^{\nu}(L(A)>0 \mid \Omega_\infty)=1$ for every $(x,m) \in S \times \Z$. This is easily seen to imply that $A$ is $\nu$-recurrent as claimed.
\end{proof}

\begin{proof}[Proof of \cref{thm:variations}] As above, it suffices to prove the claim for the space-time branching Markov process.
Fix $P$, $\mu$, and $\nu$. For each $(x,m) \in S \times \Z$, we write $\bP_{x,m}^{\mu}$ and $\bP_{x,m}^\nu$ for the laws of the space-time branching Markov process $(B^\mathrm{st}_n)_{n\geq 0}$ started at $(x,m)\in S \times \Z$ with offspring distribution $\mu$ or $\nu$ as appropriate.
Let $A_1,A_2,\ldots$ be a sequence of subsets of $S\times \Z$ and enumerate $S=\{x_1,x_2,\ldots\}$. 
%
For each $i\geq 1$ and $-\infty \leq s \leq t \leq \infty$, let $\sE_i(s,t)$ be the event that
$B^\mathrm{st}$ visits $A_i \cap (S \times [s,t])$, and let $\sR_i$ be the event that $A_i$ is visited infinitely often.
%
  For each $i\geq 1$ and $(x,m)\in S \times \Z$, we have by continuity of measure that there exists $N(i,x,m)$ such that 
\[\bP^\mu_{x,m}\left(\sE_i\bigl(N(i,x,m),\infty\bigr) \setminus \sR_i\right) \leq 2^{-i} \qquad \text{and} \qquad \bP^\nu_{x,m}\left(\sE_i\bigl(N(i,x,m),\infty\bigr) \setminus \sR_i\right) \leq 2^{-i}.\]
Let $N(i)= \max_{j\leq i,|m|\leq i} N(i,x_j,m)$ for each $i\geq 1$. It follows by Borel-Cantelli that, under any of the measures $\bP^\mu_{x,m}$ or $\bP^\nu_{x,m}$ for any $(x,m)\in S \times \Z$, the event that $I_\infty$ is infinite coincides up to a null set with the event that $B^\mathrm{st}$ visits the set
\[
\bigcup_{i\geq k} A_i \cap \bigl(S\times \{N(i),N(i)+1,\ldots\}\bigr)
\]
infinitely often for every $k\geq 1$. Applying continuity of measure a second time, we deduce that for each $i\geq 1$ and $(x,m)\in S \times \Z$ there exists $M(i,x,m) \geq N(i)$ such that
\begin{align*}\bP^\mu_{x,m}\left(\sE_i\bigl(N(i),\infty\bigr) \setminus \sE_i\bigl(N(i),M(i,x,m)\bigr)\right) &\leq 2^{-i} \qquad \text{and} \\
 \bP^\nu_{x,m}\left(\sE_i\bigl(N(i),\infty\bigr) \setminus \sE_i\bigl(N(i),M(i,x,m)\bigr)\right) &\leq 2^{-i}.\end{align*}
  Taking $M(i) = \max_{j\leq i, |m| \leq i} M(i,x_j,m)$ we deduce by Borel-Cantelli that, under any of the measures $\bP^\mu_{x,m}$ or $\bP^\nu_{x,m}$ for any $(x,m)\in S \times \Z$, the event that $I_\infty$ is infinite coincides up to a null set with the event that $B^\mathrm{st}$ visits the set
\[
\bigcup_{i\geq 1} A_i \cap \bigl(S\times \{N(i),N(i)+1,\ldots,M(i)\}\bigr)
\]
infinitely often. This reduces the claim to that of \cref{thm:germ}, completing the proof. \qedhere

\end{proof}


\section{Examples}
\label{sec:examples}

\begin{example}[Maximal displacement]
Let $G=(V,E)$ be a connected, locally finite graph and let $\mu,\nu$ be offspring distributions with $\mu \leq_\mathrm{germ} \nu$. Let $(B_n)_{n\geq0}$ be a branching random walk started at some vertex $x$ of $G$ and let $M_n = \max\{d(x,y): B_n(y) >0\}$ be the maximum displacement of a particle at time $n$. Let $f:\N\to \N$ be a function and let $\alpha>0$. We claim that if $\liminf M_n/f(n) \geq \alpha$ $\bP_x^\mu$-almost surely on the event $\Omega_\infty$ then the same is true $\bP_x^\nu$-almost surely on the event $\Omega_\infty$: this follows from the fact that $\liminf M_n/f(n) \geq \alpha$ $\bP_x^\mu$-almost surely on the event $\Omega_\infty$ if and only if the space-time set
\[
A_\eps = \Bigl\{(y,n) : n\geq 1,\, d(y,x_0) \geq (\alpha-\eps)f(n) \Bigr\}
\]
is $\mu$-recurrent for every $\eps>0$, where $x_0$ is an arbitrary fixed vertex. A similar argument yields that 
if $\limsup M_n/f(n) \leq \alpha$ $\bP_x^\nu$-almost surely then the same is true $\bP_x^\mu$-almost surely. Many works analyzing examples \cite{MR370721,MR400438,MR464415,MR1797310,MR3444654} have bound that the asymptotic rate of growth of $M_n$ depends only on the mean $\bar \mu$ under various hypotheses (it is often possible to give an exact formula), which is of course consistent with this result.
\end{example}

\begin{example}[Intersections] Let $P$ be a transition matrix on a countable state space $S$, and let $(\mu_1,\mu_2)$ and $(\nu_1,\nu_2)$ be two pairs of supercritical offspring distributions such that $\mu_i \leq_\mathrm{germ} \nu_i$ for each $i=1,2$. Suppose that if we sample two independent branching Markov processes with offspring distributions $\mu_1$ and $\mu_2$ then the sets of states visited by both processes is almost surely infinite on the event that both processes survive forever. Applying \cref{thm:germ} twice implies that the same is true if the processes are instead taken to have offspring distributions $\nu_1$ and $\nu_2$. See \cite{MR4146533} for further discussion of this property for branching random walks on transitive graphs.
\end{example}

\begin{example}[The dimension of the limit set]
\label{example:dimension}
Several authors have studied the Hausdorff dimension of the  set of limit points of branching random walk on hyperbolic spaces and groups, and have found that this dimension depends only on the mean of the offspring distribution under various hypotheses \cite{sidoravicius2020limit,MR1452555,MR1641015,MR1736590}. Let us now discuss how a similar result for the \emph{upper Minkowski dimension} can be deduced very abstractly from our main theorems. (Of course, unlike the aforementioned works, our results do not allow us to actually compute this dimension!) 

Let $P$ be a transition matrix on a countable state space $S$, let $d\geq 1$, and let $\varphi:S\to [0,1]^d$ be an embedding of the state space into $[0,1]^d$. Let $(B_n)_{n\geq0}$ and let 
\[\Gamma =\bigcap_{n=0}^\infty \overline{\{\varphi(y) : B_m(y) >0 \text{ for some $m \geq n$}\}}
\]
be the set of accumulation points. 
Given a subset $K$ of $[0,1]^d$ and $\eps>0$, let $N(\eps)$ be the minimum number of balls of radius $\eps$ needed to cover $K$.
The \textbf{upper and lower Minkowski dimensions} of $K$ are defined to be
\[
\operatorname{\overline{dim}}_M(K) = \limsup_{\eps \downarrow 0} \frac{\log N(\eps)}{\log (1/\eps)} \qquad \text{ and } \qquad  
\operatorname{\underline{dim}}_M(K) = \liminf_{\eps \downarrow 0} \frac{\log N(\eps)}{\log (1/\eps)}.
\]
See e.g.\ \cite{MR3236784} for background. We claim that if $\mu \leq_\mathrm{germ} \nu$ then the following hold:
\begin{enumerate}
	\item If $\alpha$ is such that $\operatorname{\overline{dim}}_M(\Gamma) > \alpha$ $\bP^\mu_x$-almost surely on the event $\Omega_\infty$ for every $x\in S$ then $\operatorname{\overline{dim}}_M(\Gamma) \geq \alpha$ $\bP^\nu_x$-almost surely on the event $\Omega_\infty$  for every $x\in S$.
	\item If $\alpha$ is such that $\operatorname{\overline{dim}}_M(\Gamma) < \alpha$ $\bP^\nu_x$-almost surely for every $x\in S$ then $\operatorname{\overline{dim}}_M(\Gamma) \leq \alpha$ $\bP^\mu_x$-almost surely for every $x\in S$.
\end{enumerate}
Indeed, for each $k\geq 1$ let $D_k$ be the set of closed dyadic subcubes of $[0,1]^d$, which have side length $2^{-k}$ and corner coordinates that are integer multiples of $2^{-k}$. For each $0<\alpha\leq d$ and $k\geq 1$ let $R_k(\alpha)$ be the union of the cubes in the random subset of $D_k$ in which each possible cube is chosen to be included or not independently at random with inclusion probability $2^{-\alpha k}$. We take the sequence of random variables $(R_k(\alpha))_{k\geq 1}$ to be independent of each other. For each $k\geq 1$, let $R_k'(\alpha)$ be obtained by enlarging each box contributing to $R_k(\alpha)$ by a factor of $2$, so that $R_k(\alpha)'$ contains an open neighbourhood of $R_k(\alpha)$. It is an easy consequence of (both parts of) the Borel-Cantelli lemma that if $K \subseteq [0,1]^d$ is a fixed closed set then we have the implications
\begin{align*}
\Bigl(\operatorname{\overline{dim}}_M(K) > \alpha\Bigr) &\Rightarrow
\Bigl(K \cap R_k(\alpha) \neq \emptyset  \text{ infinitely often a.s.\ as $k\to\infty$}\Bigr) \\
&\Rightarrow
\Bigl(K \cap R_k'(\alpha) \neq \emptyset  \text{ infinitely often a.s.\ as $k\to\infty$}\Bigr)&\Rightarrow \Bigl(\operatorname{\overline{dim}}_M(K) \geq \alpha\Bigr).
\end{align*}
Take $K=\Gamma$ to be independent of the random sets $(R_k(\alpha))_{k\geq 1}$, and observe that $\Gamma \cap R_k(\alpha) \neq \emptyset$ if and only if $B$ visits $\varphi^{-1}(U)$ infinitely often for every open neighbourhood $U$ of $R_k(\alpha)$.
In particular, if 
$\operatorname{\overline{dim}}_M(\Gamma) > \alpha$ $\bP^\mu_x$-almost surely on the event $\Omega_\infty$ for every $x\in S$ then, on the event $\Omega_\infty$, there are almost surely infinitely many $k$ such that $R'_k(\alpha)$ is visited infinitely often by $B$, and it follows from \cref{thm:variations} that the same is true with respect to the offspring distribution $\nu$. (Applying this theorem in the case that the $A_i$ are themselves random can be done by a standard Fubini argument.)
Similarly, if 
$\operatorname{\overline{dim}}_M(\Gamma) < \alpha$ $\bP^\nu_x$-almost surely for every $x\in S$ then we must have that there are almost surely at most finitely many $k$ such that $R'_k(\alpha)$ is visited infinitely often by $B$, and it follows from \cref{thm:variations} that the same is true with respect to the offspring distribution $\mu$.
\end{example}

\subsection*{Acknowledgments}
We thank Toby Johnson and Matt Junge for helpful discussions.

 \setstretch{1}
 \footnotesize{
  \bibliographystyle{abbrv}
  \bibliography{unimodularthesis.bib}
  }

\end{document}